\documentclass{article}

\usepackage{tikz}

\usepackage{amsthm}
\usepackage{amsmath}
\usepackage{amssymb}
\usepackage{enumerate}

\newtheorem{theorem}{Theorem}[section]
\newtheorem{lemma}[theorem]{Lemma}
\newtheorem{proposition}[theorem]{Proposition}
\newtheorem{corollary}[theorem]{Corollary}

\numberwithin{equation}{section}

\makeatletter
\newcommand{\nolisttopbreak}{\vspace{\topsep}\nobreak\@afterheading}
\makeatother
\newenvironment{listproof}[1][\proofname]{\begin{proof}[#1]\mbox{}\nolisttopbreak}{\end{proof}}

\newcommand{\oth}{\textsuperscript{th}\ }
\newcommand{\jac}{{\mathcal{J}}}
\newcommand{\hess}{{\mathcal{H}}}
\newcommand{\Mat}{\operatorname{Mat}}
\newcommand{\GL}{\operatorname{GL}}
\newcommand{\rk}{\operatorname{rk}}
\newcommand{\trdeg}{\operatorname{trdeg}}
\newcommand{\tp}{^{\rm t}}
\newcommand{\sdots}{\vdots\,\vdots\,\vdots}
\newcommand{\imp}{{\mathversion{bold}$\Rightarrow$} }
\newcommand{\zeromat}{\mbox{\begin{tikzpicture} \useasboundingbox (-0.1,-0.125) -- (0.1,0.125);
\draw (-0.05,0.05) -- node[anchor=center]{\rm 0} (0.05,-0.05); \end{tikzpicture}}}

\title{Quadratic polynomial maps with Jacobian rank two}
\author{Michiel de Bondt\footnote{The author was supported by the Netherlands 
        Organisation for Scientific Research (NWO).} \\
        Institute for Mathematics, Astrophisics and Particle Physics \\
        Radboud University Nijmegen \\
        \emph{Email address:} M.deBondt@math.ru.nl}

\begin{document}

\maketitle

\begin{abstract}
Let $K$ be any field and $x = (x_1,x_2,\ldots,x_n)$. 
We classify all matrices $M \in \Mat_{m,n}(K[x])$ whose entries 
are polynomials of degree at most 1, for which $\rk M \le 2$. As a special case, we
describe all such matrices $M$, which are the Jacobian matrix $\jac H$ (the matrix of
partial derivatives) of a polynomial map $H$ from $K^n$ to $K^m$. 

Among other things, we show that up to composition with linear maps over $K$, 
$M = \jac H$ has only two nonzero columns or only three nonzero rows in this case.
In addition, we show that $\trdeg_K K(H) = \rk \jac H$ for quadratic polynomial 
maps $H$ over $K$ such that $\frac12 \in K$ and $\rk \jac H \le 2$.

Furthermore, we prove that up to conjugation with linear maps over $K$,
nilpotent Jacobian matrices $N$ of quadratic polynomial maps, for which $\rk N \le 2$, 
are triangular (with zeroes on the diagonal), regardless of the characteristic of $K$. 
This generalizes several results by others.

In addition, we prove the same result for Jacobian matrices $N$ of quadratic 
polynomial maps, for which $N^2 = 0$. This generalizes a result by others, 
namely the case where $\frac12 \in K$ and $N(0) = 0$.
\end{abstract}



\section{Introduction}

Throughout this paper, $K$ is an arbitrary field and $x := (x_1,x_2,\ldots,x_n)$ 
is an $n$-tuple of indeterminates. 
We write $a|_{b=c}$ for the result of substituting $b$ by $c$ in $a$.

We call a polynomial $h \in K[x]$
\emph{homogeneous (of degree $d$)} if all terms of $h$ have the same degree
(and $\deg h = d$). We call the terms of degree $\deg h$ of $h$ the 
\emph{leading homogeneous part} of $h$. A \emph{linear form} is a polynomial
$h \in K[x]$ which is homogeneous of degree $1$.

Let $H \in K[x]^m$. Then $H = (H_1,H_2,\ldots,H_m)$ is a polynomial map 
from $K^n$ to $K^m$. The \emph{degree} of $H$ is defined by 
$\deg H := \max\{\deg H_1, \deg H_2, \ldots, \allowbreak \deg H_m\}$. 
We say that $H$ is \emph{homogeneous (of degree $d$)} if
$y_1 H_1 + y_2 H_2 + \cdots + y_m H_m$ is a homogeneous polynomial (of 
degree $d + 1$).

We write $\jac H$ for the \emph{Jacobian matrix of $H$} (with respect to $x$), i.e.
$$
\jac H = \left( \begin{array}{cccc}
\frac{\partial}{\partial x_1} H_1 & \frac{\partial}{\partial x_2} H_1 & 
                           \cdots & \frac{\partial}{\partial x_n} H_1 \\
\frac{\partial}{\partial x_1} H_2 & \frac{\partial}{\partial x_2} H_2 & 
                           \cdots & \frac{\partial}{\partial x_n} H_2 \\
\vdots & \vdots & \sdots\, & \vdots \\
\frac{\partial}{\partial x_1} H_m & \frac{\partial}{\partial x_2} H_m & 
                           \cdots & \frac{\partial}{\partial x_n} H_m \\
\end{array} \right)
$$
We call a matrix a \emph{Jacobian matrix} if it is the Jacobian matrix 
of some polynomial map. Note that this polynomial map is not 
uniquely determined if our base field $K$ has positive characteristic.

Let $R$ be a commutative ring with $1$. We write $\Mat_{m,n}(R)$ for the abelian 
group of matrices with $m$ rows and $n$ columns over $R$. So $\jac H \in \Mat_{m,n}(K[x])$.
We write $\Mat_n(R)$ for the ring of matrices with $n$ rows and $n$ columns 
over $R$. We define $\GL_n(R)$ as the group of invertible matrices in $\Mat_n(R)$, 
i.e.\@ $\GL_n(R) := \{ M \in \Mat_n(R) \mid \det M$ is a unit in $R\}$. 

If $R$ is a $K$-algebra, then we say that elements $M$ and $\tilde{M}$ 
of $\Mat_{m,n}(R)$ are \emph{equivalent over $K$} if there exists matrices 
$S \in \GL_m(K)$ and $T \in \GL_n(K)$ such that $\tilde{M} = SMT$. 
If $m = n$ and $S = T^{-1}$ in addition, then we say that $M$ and $\tilde{M}$ are
\emph{similar over $K$}. 

A matrix $M \in \Mat_n(R)$ is \emph{upper (lower) triangular} if all entries below 
(above) the principal diagonal are zero, and \emph{triangular (diagonal)} if $M$ is 
either (both) upper or (and) lower triangular. So a diagonal matrix may only have
nonzero entries on the diagonal which runs from its upper left corner to its lower
right corner. This diagonal is called \emph{the (principal) diagonal}.

We say that a matrix $M \in \Mat_n(R)$ is \emph{nilpotent} if the exists an $r \ge 1$, 
such that $M^r = 0$. The reader may verify the following.

\begin{lemma}
Suppose that $R$ is a $K$-algebra and $M \in \Mat_n(R)$. 
\begin{enumerate}

\item[(i)]
If $M$ is upper (lower) triangular, then $M$ is similar over $K$ to a lower 
(upper) triangular matrix $\tilde{M} \in \Mat_n(R)$

\item[(ii)]
If $\tilde{M}$ is similar over $K$ to $M$, then $\tilde{M}$ is nilpotent, 
if and only if $M$ is nilpotent.

\end{enumerate}
\end{lemma}

If $M \in \Mat_{m,n}(K[x])$, then we write $M(v)$
for the matrix
$$
\left( \begin{array}{cccc}
M_{11}(v) & M_{12}(v) & \cdots & M_{1n}(v) \\
M_{21}(v) & M_{22}(v) & \cdots & M_{2n}(v) \\
\vdots & \vdots & \sdots\, & \vdots \\
M_{m1}(v) & M_{m2}(v) & \cdots & M_{mn}(v) 
\end{array} \right)
$$
where $v \in K^n$. We say that $M \in \Mat_m(K[x])$ is \emph{strongly nilpotent
(over $K$)} if there exists an $r \ge 1$, such that
$$
M(v^{(1)}) \cdot M(v^{(2)}) \cdot \cdots \cdot M(v^{(r)}) = 0
$$
for all $v^{(1)},v^{(2)},\ldots,v^{(r)} \in K^n$. If $K$ is infinite, 
then proposition \ref{stronglt} below gives a classification of strongly nilpotent 
matrices over $K[x]$. For the proof of proposition \ref{stronglt}, we need the 
following lemma, which one can show by induction on $r$.

\begin{lemma} \label{strongblock}
Suppose that $M \in \Mat_m(K[x])$ is of the form
$$
\left( \begin{array}{cc} A & \zeromat \\ * & B \end{array} \right)
\qquad \mbox{or} \qquad
\left( \begin{array}{cc} A & * \\ \zeromat & B \end{array} \right)
$$
Let $\tilde{M} := M(v^{(1)}) \cdot M(v^{(2)}) \cdot \cdots \cdot M(v^{(r)})$,
$\tilde{A} := A(v^{(1)}) \cdot A(v^{(2)}) \cdot \cdots \allowbreak \cdot A(v^{(r)})$ 
and $\tilde{B} := B(v^{(1)}) \cdot B(v^{(2)}) \cdot \cdots \cdot B(v^{(r)})$,
where $v^{(i)} \in K[x]^n$ for each $i$. Then $\tilde{M}$ is of the form.
\begin{equation} \label{strongeq}
\left( \begin{array}{cc} \tilde{A} & \zeromat \\ * & \tilde{B} \end{array} \right)
\qquad \mbox{or} \qquad
\left( \begin{array}{cc} \tilde{A} & * \\ \zeromat & \tilde{B} \end{array} \right)
\end{equation}
respectively.
\end{lemma}

\begin{proposition} \label{stronglt}
Suppose that $L$ is infinite and an extension field of $K$. Let $M \in \Mat_m(K[x])$. 
Then $M$ is strongly nilpotent over $L$, if and only if $M$ is similar over $K$ to a 
triangular matrix in $\Mat_m(K[x])$, whose principal diagonal is totally zero.
\end{proposition}

\begin{proof}
The `if'-part is an straightforward exercise, so assume that $M$ is strongly nilpotent 
over $L$. Then there are $v^{(2)}, v^{(3)}, \ldots, v^{(r)} \in L^n$ such that for all 
$v^{(1)} \in L^n$,
$$
M(v^{(1)}) \cdot M(v^{(2)}) \cdot \cdots \cdot M(v^{(r)}) = 0 \ne
M(v^{(2)}) \cdot M(v^{(3)}) \cdot \cdots \cdot M(v^{(r)})
$$
Since $L$ is infinite, it follows that 
$$
M \cdot M(v^{(2)}) \cdot \cdots \cdot M(v^{(r)}) = 0 \ne
M(v^{(2)}) \cdot M(v^{(3)}) \cdot \cdots \cdot M(v^{(r)})
$$
so the columns of $M$ are linearly dependent over $L$. Since $L$ is a vector space over 
$K$, the columns of $M$ are linearly dependent over $K$. Hence $M$ is similar over $K$
to a matrix $\tilde{M} \in \Mat_m(K[x])$, of which the last column is zero. 

From lemma \ref{strongblock}, it follows that the upper left submatrix of 
size $(m-1) \times (m-1)$ of $\tilde{M}$ strongly nilpotent. By induction on 
$m$, it follows that $\tilde{M}$ and hence also $M$ is similar over $K$ 
to a lower triangular matrix in $\Mat_m(K[x])$, whose principal diagonal is totally 
zero.
\end{proof}

The above proof has been extracted from that of \cite[Th.~3.1]{MR3177043}, 
which is a more general result.

\begin{corollary} \label{blocklt}
Suppose that $M \in \Mat_m(K[x])$ is of the form
$$
\left( \begin{array}{cc} A & \zeromat \\ * & B \end{array} \right)
\qquad \mbox{or} \qquad
\left( \begin{array}{cc} A & * \\ \zeromat & B \end{array} \right)
$$
Then $M$ is similar over $K$ to a triangular matrix with only zeroes
on its principal diagonal, if and only if $A$ and $B$ are similar over $K$
to triangular matrices with only zeroes on their principal diagonals.
\end{corollary}

\begin{proof}
From lemma \ref{strongblock}, it follows that $A$ and $B$ are strongly nilpotent over 
an infinite extension field $L$ of $K$ if $M$ is strongly nilpotent over $L$.
Hence the `only if'-part follows from proposition \ref{stronglt}.

To prove the `if'-part, suppose that $A$ and $B$ are strongly nilpotent over 
an infinite extension field $L$. Then there exists an integer $r$, such that
$\tilde{A} = \tilde{B} = 0$ in \eqref{strongeq}, for every $v^{(1)}, v^{(2)},
\ldots, v^{(r)} \in L^n$. It follows that 
$$
\big(M(v^{(1)}) \cdot M(v^{(2)}) \cdot \cdots \cdot M(v^{(r)})\big) \cdot
\big(M(v^{(r+1)}) \cdot M(v^{(r+2)}) \cdot \cdots \cdot M(v^{(2r)})\big) = 0
$$
for all $v^{(1)},v^{(2)},v^{(3)},\ldots,v^{(2r-1)},v^{(2r)} \in K^n$. Hence the
`if'-part follows from proposition \ref{stronglt} as well.
\end{proof}

A \emph{minor (determinant)} of a matrix $M$ is the determinant of a square
submatrix, say $N$, of $M$. The submatrix $N$ itself is called a 
\emph{minor matrix}. If the entries of the (principal) diagonal of $N$ lie
on the (principal) diagonal of $M$ as well, then $\det N$ and $N$ are called a
\emph{principal minor (determinant)} and a \emph{principal minor matrix}
respectively. 

Notice that a minor matrix $N$ of $M$ is a principal minor matrix of $M$, if and only if
the indices of the rows of $M$ which $N$ has as a submatrix of $M$ are the same as 
the indices of the columns of $M$ which $N$ has as a submatrix. 

We call a principal minor matrix of $M$ a \emph{leading principal minor matrix} if it extends
from the upper left corner of $M$, and a \emph{trailing principal minor matrix} if it extends
from the lower right corner of $M$. Similarly, we define \emph{leading principal minor 
(determinant)} and \emph{trailing principal minor (determinant)}.

We call $\lambda$ an \emph{eigenvalue} of $M \in \Mat_m(K)$ if there exists
a nonzero $v \in K^m$ such that $M v = \lambda v$. Notice that we must view 
the vector $v$ as a matrix with only one column in $M v = \lambda v$. We will
regard vectors as matrices with only one column in the rest of the paper as well.

\begin{lemma} \label{eigen}
Suppose that $M \in \Mat_m(K)$. Then the following statements are equivalent.
\begin{enumerate}

\item[(1)] $M$ is nilpotent;

\item[(2)] every eigenvalue of $M$ equals zero;

\item[(3)] for every $r \le m$, the sum of the principal minor determinants
of size $r \times r$ is zero.

\end{enumerate}
\end{lemma}

\begin{listproof}
\begin{description}

\item[(1) \imp (2)] Suppose that $\lambda$ is an eigenvalue of $M$.
Then $M$ is similar over $K$ to a matrix $\tilde{M}$ for which 
$\tilde{M} e_1 = \lambda e_1$. Since $\tilde{M}$ is nilpotent as well,
it follows from lemma \ref{strongblock} that $\lambda = 0$.

\item[(2) \imp (3)] Let $f(u) = \det (u I_m + M)$. Then
for every $r$ for which $1 \le r \le m$, the coefficient of
$u^{m-r}$ in $f$ equals the sum of the principal minor determinants
of size $r \times r$ of $M$. 

If $f(-\lambda) = 0$, then $\ker (-\lambda I_m + M) \ne \{0\}$, 
and $\lambda$ is an eigenvalue of $M$. Suppose that (2) holds. 
Then $u = -0$ is the only root of $f$. 
Hence $f = u^m$ and (3) follows.

\item[(3) \imp (1)]
Suppose that (3) holds. Then $f = u^m$. From the Cayley-Hamilton
theorem, it follows that $(-M)^m = 0$, which gives (1). \qedhere

\end{description}
\end{listproof}

Suppose that $M = \jac H$ and $\tilde{M} = SMT$, where $H$ is a polynomial map 
from $K^n$ to $K^m$, $S \in \GL_m(K)$ and $T \in \GL_n(K)$. 
Let $\tilde{H} := SH(Tx)$. From the chain rule, it follows that
\begin{equation} \label{chain}
\jac \tilde{H} = \jac \big(SH(Tx)\big) = SM|_{x=Tx}T = \tilde{M}|_{x=Tx}
\end{equation}
so $\tilde{M}$ itself is a Jacobian matrix up to an automorphism of $K[x]$.
It follows that $\jac \tilde{H}$ is (strongly) nilpotent or upper (lower) 
triangular, if and only if $\tilde{M}$ is (strongly) nilpotent or upper 
(lower) triangular respectively.

The \emph{degree} of a matrix $M \in \Mat_{m,n}(K[x])$ is defined by
$$
\deg M := \max \{ \deg M_{11}, \deg M_{12} ,\ldots, \deg M_{1n},
\deg M_{21}, \deg M_{22}, \ldots, \deg M_{mn} \}
$$
and we write $\rk M$ for the rank of $M$.

Let $M \in \Mat_{m,n}(R)$, where $R$ is a ring. Then we denote
by $M\tp$ the \emph{transpose} of $M$, so
$$
M\tp = \left( \begin{array}{cccc}
M_{11} & M_{12} & \cdots & M_{1n} \\
M_{21} & M_{22} & \cdots & M_{2n} \\
\vdots & \vdots & \sdots\, & \vdots \\
M_{m1} & M_{m2} & \cdots & M_{mn} 
\end{array} \right)\tp := \left( \begin{array}{cccc}
M_{11} & M_{21} & \cdots & M_{m1} \\
M_{12} & M_{22} & \cdots & M_{m2} \\
\vdots & \vdots & \sdots\, & \vdots \\
M_{1n} & M_{2n} & \cdots & M_{mn} 
\end{array} \right)
$$
Notice that the symbol t is upright to distinguish from taking the $t$\oth power.
We call $M$ \emph{symmetric} if $M\tp = M$ and \emph{antisymmetric} if $M\tp = -M$.

Suppose that $K$ is a subfield of a field $L$. Then we write $\trdeg_K L$
for the \emph{transcendence degree} of $L$ over $K$, i.e.\@
$$
\trdeg_K L = \max\{\# S \mid S \subseteq L \mbox{ and $S$ is algebraically independent over $K$} \,\} 
$$
Here, a subset $S$ of $L$ is \emph{algebraically independent over $K$} if the result
of substituting elements of $S$ in a nonzero polynomial in finitely many variables
over $K$ will never be zero.

\paragraph{\S \ref{srk2}}
In section \ref{srk2}, we classify all matrices $M \in \Mat_{m,n}(K[x])$ 
for which $\deg M = 1$ and $\rk M \le 2$. Furthermore, we classify all 
such matrices $M$ such that $M = \jac H$ for some polynomial map $H$.
Among other things, we show that $M = \jac H$ is equivalent over $K$ 
to a matrix $\tilde{M}$ which has either only two nonzero columns or 
only three nonzero rows.

In addition, we show that $\rk \jac H = \trdeg_K K(H)$ for quadratic polynomial 
maps $H$ over $K$ such that $\frac12 \in K$ and $\rk \jac H \le 2$.
In general $\rk \jac H \le \trdeg_K K(H)$ for a rational map $H$ of any 
degree, with equality if $K(H) \subseteq K(x)$ is separable, in particular if
$K$ has characteristic zero. This is proved in \cite[Th.\@ 1.3]{1501.06046},
see also \cite[Ths.\@ 10, 13]{DBLP:conf/mfcs/PandeySS16}.

\paragraph{\S \ref{srk2n}}
In section \ref{srk2n}, we prove that nilpotent Jacobian matrices $N$
for which $\deg N = 1$ and $\rk N \le 2$ are similar over $K$ to
a triangular matrix (with zeroes on the diagonal), regardless of
the characteristic of $K$. This generalizes \cite[Th.~3.4]{MR3271211}
(the case where $K$ has characteristic zero) and \cite[Th.~1]{MR3327129}
(the case where $\frac12 \in K$ and $N(0) = 0$).
In \cite[Th.~1]{MR3327129}, which is the main result of \cite{MR3327129}, the 
authors additionally assume that $K$ is infinite, but one can derive the finite
case from the infinite case by way of proposition \ref{stronglt} above.

At the end of section \ref{srk2n}, we prove that nilpotent Jacobian matrices $N$
for which $\deg N = 1$ and $N^2=0$ are similar over $K$ to
a triangular matrix (with zeroes on the diagonal), regardless of
the characteristic of $K$. This generalizes \cite[\S 4]{MR1129181} and 
\cite[Lm.~4]{MR3327129} (the case where $\frac12 \in K$ and $N(0) = 0$).

We additionally show that $N(v^{(1)}) \cdot N(v^{(2)}) \cdot N(v^{(3)}) = 0$ 
if $\frac12 \in K$,  where $v^{(1)},v^{(2)},v^{(3)}$ are as in the proof of 
proposition \ref{stronglt}, using the fact that the proof of 
\cite[Lm.~4]{MR3327129} shows that $N(v^{(1)}) \cdot N(v^{(2)}) = 0$ 
if $N(0) = 0$ in addition.

\section{(Jacobian) matrices of degree one and rank at most two} \label{srk2}

A matrix of rank zero can only be the zero matrix, so we only need to 
distinguish rank one and rank two. Let us start with rank one.

\begin{theorem} \label{quadrk1}
Let $M$ be a matrix whose entries are polynomials of degree at most $1$ over $K$.  
If $\rk M = 1$, then $M$ is equivalent over $K$ to a matrix $\tilde{M}$ 
for which one of the following statements holds.
\begin{enumerate}[\upshape\bf (1)]

\item Only the first column of $\tilde{M}$ is nonzero.

\item Only the first row of $\tilde{M}$ is nonzero.

\end{enumerate}
If $M$ is the Jacobian matrix of a (quadratic) polynomial map in addition, 
then the following assertion can be added to {\upshape (1)}.
\begin{enumerate}[\upshape\bf (1)]

\item The first column of $\tilde{M}$ is of the form $(*,\frac12,0,\ldots,0)$.

\end{enumerate}
\end{theorem}
 
\begin{proof}
If the constant part $M(0)$ of $M$ is zero, then we can replace
$M$ by the result of substituting $x_i = x_i + 1$ in $M$,
where $x_i$ is an indeterminate which appears in $M$, 
to obtain $M(0) \ne 1$. So we may assume that $M(0) \ne 0$.

Since $\rk M(0) = 1$, we can choose $\tilde{M}$ such that 
$$
\tilde{M}(0) = \left( \begin{array}{cccc}
1 & 0 & 0 & \cdots \\
0 & 0 & 0 & \cdots \\
0 & 0 & 0 & \cdots \\
\vdots & \vdots & \vdots & \sdots
\end{array} \right)
$$
Looking at the linear parts of the $2 \times 2$ minor determinants, 
we see that only the first row and the first column of $\tilde{M}$ may be nonzero. 

Suppose that (1) does not hold.
Then we may assume that the second entry of the first row of $\tilde{M}$ is nonzero.
Suppose that (2) does not hold.
Then we may assume that the first entry of the second row of $\tilde{M}$ is nonzero.
This contradicts that the leading principal $2 \times 2$ minor determinant
of $\tilde{M}$ is zero.

So we have proved the first part of this theorem. To prove the second part
of this theorem, assume that $M = \jac H$ for a polynomial
map $H$. If we remove terms $x_1^{k_1}x_2^{k_2}\cdots$ of $H$ for which 
$\frac1{k_i} \notin K$ for all $i$, then $M = \jac H$ is preserved, and
for every term $t$ of $H$, there exists an $i$ such that
$$
\frac{\partial}{\partial x_i} t \ne 0 
$$
Since $\deg \jac H \le 1$, it follows that $H$ becomes a polynomial map with terms 
of degree $1$ and $2$ only. So we may assume that $H(0) = 0$ and $\deg H \le 2$.

Say that $\tilde{M} = SMT$ for invertible matrices $S, T$ over $K$.
Let $\tilde{H} := SH(Tx)$ and suppose that $\tilde{M}$ is as in (1). 
From \eqref{chain}, it follows that $\jac \tilde{H}$ is as in (1) as well, 
i.e.\@ only the first column of $\jac \tilde{H}$ is nonzero.

If $\frac12 \in K$, then for all $j$, $\tilde{H}_j$ is linearly
dependent over $K$ on $x_1^2$ and $x_1$ only. If $x_1$ and $x_1^2$ are in turn
linearly dependent over $K$ on $\tilde{H}_1, \tilde{H}_2, \cdots$, then we can 
get the first column of $\jac \tilde{H}$ and $\tilde{M}$ of the given form by 
way of row operations.
Otherwise, we can get the first column of $\jac \tilde{H}$ and $\tilde{M}$
of the form $(*,0,0,\cdots,0)$ by way of row operations, so (2) is satisfied.

So assume that $\frac12 \notin K$. Then for all $j$, $\tilde{H}_j$ is dependent 
over $K$ on $x_1, x_1^2, x_2^2, \allowbreak x_3^2, \ldots$. Hence the first
column of $\jac \tilde{H}$ is constant. By way of row operations, we can get 
the first column of $\jac \tilde{H}$ and $\tilde{M}$ of the form $(*,0,0,\cdots,0)$. 
Hence only the first row of $\jac \tilde{H}$ is nonzero and $\tilde{M}$ is as in (2).
\end{proof}

In \cite[Th.~1.8]{MR3179983}, it is proved that over fields of 
characteristic zero, polynomial maps with an antisymmetric Jacobian matrix are 
linear. With essentially the same proof, one can draw the same conclusion if 
the characteristic of the field exceeds the degree of the polynomial map.

\begin{lemma} \label{antisym}
Let $H$ be a polynomial map of degree at most $d$ over $K$, such that 
$d! \ne 0$ in $K$. If $\jac H$ is antisymmetric, then $\deg H \le 1$.
\end{lemma}

\begin{proof}
There is nothing to prove if $\frac12 \notin K$, so assume that $\frac12 \in K$.
Suppose that $\jac H$ is antisymmetric. Then
$$
\frac{\partial}{\partial x_i}\frac{\partial}{\partial x_j} H_k 
= - \frac{\partial}{\partial x_i}\frac{\partial}{\partial x_k} H_j 
= \frac{\partial}{\partial x_j} \frac{\partial}{\partial x_k} H_i 
= -\frac{\partial}{\partial x_i}\frac{\partial}{\partial x_j} H_k
$$
and hence $2 \frac{\partial}{\partial x_i}\frac{\partial}{\partial x_j} H_k = 0$, 
for all $i,j,k$. As $2d! \ne 0$ in $K$, it follows that $\deg H \le 1$. 
\end{proof}

Using lemma \ref{antisym} above, we can proceed with rank two.

\begin{theorem} \label{quadrk2}
Let $M$ be a matrix whose entries are polynomials of degree at most $1$ over $K$. 
If $\rk M = 2$, then $M$ is equivalent over $K$ to a matrix $\tilde{M}$ for 
which one of the following statements holds.
\begin{enumerate}[\upshape\bf (1)]

\item Only the first two columns of $\tilde{M}$ are nonzero.

\item Only the first two rows of $\tilde{M}$ are nonzero.

\item The first row and the first column of $\tilde{M}$ are nonzero, and 
$\tilde{M}$ is zero elsewhere.

\item The leading principal $3 \times 3$ minor matrix of $\tilde{M}$
is anti-symmetric, with only zeroes on the diagonal, 
and $\tilde{M}$ is zero elsewhere. Furthermore, the three entries below
the diagonal of this principal minor matrix are linearly independent over $K$.

\end{enumerate}
If $M$ is the Jacobian matrix of a (quadratic) polynomial map in addition, 
then the following assertions can be added to {\upshape (3)} and 
{\upshape (4)} respectively.
\begin{enumerate}[\upshape\bf (1)]
\addtocounter{enumi}{2}

\item The first column of $\tilde{M}$ is of the form $(*,*,\frac12,0,\ldots,0)$.

\item $\tilde{M}$ is symmetric, i.e.\@ $\frac12 \notin K$. 

\end{enumerate}
\end{theorem}

\begin{proof}
We first show that we may assume that 
\begin{equation} \label{cmx1}
M(0) = \left( \begin{array}{cccc}
0 & -1 & 0 & \cdots \\
1 & 0 & 0 & \cdots \\
0 & 0 & 0 & \cdots \\
\vdots & \vdots & \vdots & \sdots
\end{array} \right)
\end{equation}
Notice that this is indeed the case if $\rk M(0) \ge 2$. Otherwise, we
can follow the proof of theorem \ref{quadrk1}, to deduce that we
may assume that $\rk M(0) = 1$. So we may assume that 
$$
M(0) = 
\left( \begin{array}{cccc}
0 & 0 & 0 & \cdots \\
1 & 0 & 0 & \cdots \\
0 & 0 & 0 & \cdots \\
\vdots & \vdots & \vdots & \sdots
\end{array} \right)
$$
Suppose that an $\tilde{M}$ as in (3) cannot be obtained by interchanging the first 
and the second row of $M$. Then $M$ has a nonzero entry outside the second row and 
the first column. Without loss of generality, we may assume that the second entry 
of the first row of $M$ is nonzero. 

Say that the coefficient of $x_j$ of this entry of $M$ is nonzero. Let $C$ be the 
coefficient matrix of $x_j$ of $M$. Then
$$
x_j^{-1} \big(M - x_j C\big) + x_j^0 C
$$
is the expansion of the matrix $x_j^{-1}M$ to powers of $x_j$ (which may be negative),
because $M - x_j C$ and $C$ are the coefficient matrices of $x_j^{-1}$ and $x_j^0$
of $x_j^{-1}M$ respectively. Let $C^{*}$ be the matrix one gets by substituting 
$x_j = x_j^{-1}$ in $x_j^{-1}M$. Then
$$
C^{*} = x_j^{+1} \big(M - x_j C\big) + x_j^0 C = x_j \big(M - M(0) - x_j C\big) + x_j M(0) + C
$$
If $\rk C \ge 2$, then we can 
interchange $C$ and $M(0)$ as coefficient matrices of $M$ without affecting 
$\rk M = 2$, namely by replacing $M$ by the result of substituting $x_i = x_j^{-1}x_i$ 
for all $i \ne j$ in $C^{*}$, to obtain $\rk M(0) \ge 2$ as above. 

So assume that $\rk C = 1$. Since the second entry of the first row of $C$
is nonzero, we may assume that the first row of $C$ equals 
$(0~1~0~0~\cdots~0)$. Since $\rk C = 1$, it follows that only the second column of 
$C$ is nonzero, so we may assume that $C$ is the transpose of
$M(0)$. Now replace $M$ by the result of substituting $x_j = x_j - 1$ in $M$,
to obtain that $M(0)$ is as in \eqref{cmx1}.

So $M(0)$ is as in \eqref{cmx1}. 
Looking at the linear parts of the $3 \times 3$ minor determinants, we see that 
only the first two rows and the first two columns of $M$ may be nonzero. 
Take $i \ge 3$ arbitrary.

Suppose that (2) does not hold.
Then we may assume that the third row of $M$ is nonzero.
Suppose first that the first entry of the third row of $M$ is zero. 
Then the second entry of the third row cannot be zero. Now we can replace $M$ 
by the result of adding the second column and the second row to the first 
column and the first row respectively in $M$, to make the first entry of the 
third row of $M$ nonzero as well, without affecting $M(0)$.

So we may assume that the first entry of the third row of $M$ is nonzero.
Looking at the quadratic part of the leading principal $3 \times 3$
minor determinant, we see that the third column
is dependent on the transpose of the third row. If $i > 3$, then
we could interchange the third and the $i$\oth column of $M$, so the $i$\oth
column of $M$ is dependent on the transpose of the third row.

Suppose that (1) does not hold. 
Then we may assume that the third column of $M$ is nonzero. Just like the 
$i$\oth column of $M$ is dependent on the transpose of the third row,
we can deduce that the $i$\oth row of $M$ is dependent on the transpose
of the third column. So the $i$\oth row of $M$ is dependent on the
third row. In addition, the $i$\oth column of $M$ is dependent on the
third column.

Since the third column
is dependent on the transpose of the third row, the third entry of
the first row of $M$ is nonzero along with the first entry of the third row. 
Furthermore, we may assume that the leading principal $3 \times 3$ matrix of 
$M$ is of the form
\begin{equation} \label{Mform}
\left(\begin{array}{ccc} {*} & {*} & -a \\ {*} & c & -b \\ a & b & 0 
\end{array} \right) \qquad \mbox{or} \qquad \left(\begin{array}{ccc}
{*} & {*} & b \\ {*} & c & \lambda b \\ a & \lambda a & 0 \end{array} \right)
\end{equation}
where $a$, $b$ and $c$ are linear forms over $K$ and $\lambda \in K$.

Suppose first that the leading principal $3 \times 3$ matrix of $M$ is 
of the form of the rightmost matrix of \eqref{Mform}. 
Then we can replace $M$ by the result of subtracting the first 
column and the first row $\lambda$ times from the second column and the second row
respectively in $M$, to obtain $\lambda = 0$ in the rightmost matrix of 
\eqref{Mform}, without affecting $M(0)$. After that, we can look at the leading 
principal $3 \times 3$ minor determinant to deduce that $c$ has become zero.
So (3) is satisfied.

Suppose next that the leading principal $3 \times 3$ matrix of $M$ is of 
the form of the leftmost matrix of \eqref{Mform}, 
but not of the form of the rightmost matrix of \eqref{Mform}. Then
$a$ and $b$ are independent linear forms. 
Looking at the leading principal $3 \times 3$ minor determinant, we see that 
$b \mid a^2 c$, so $b \mid c$ and $c = \mu b$ for some $\mu \in K$. 
Now replace $M$ by the result of subtracting the third row $\mu$ times from the 
second row, to obtain $\mu = 0$, without affecting $M(0)$. 
So we may assume that $c = 0$ in the leftmost matrix of \eqref{Mform}. 
In a similar manner, we can clean the upper left corner of $M$, so
we may assume that the diagonal of the leftmost matrix of 
\eqref{Mform} is zero.

Now it is straightforward to check that the leading principal $3 \times 3$ 
matrix of $M$ is antisymmetric. Furthermore, the three entries below its diagonal
are linearly independent over $K$, because $a$, $b$ and $f+1$ are linearly
independent over $K$ for every linear form $f$.
Since $a$ and $b$ are linearly independent, it follows that the $i$\oth row of $M$
is dependent over $K$ on the third row. 
Since $-a$ and $-b$ are linearly independent, it follows that the $i$\oth column of $M$
is dependent over $K$ on the third column. 
So we can make $\tilde{M}$ as in (4) from $M$ by way of row 
and column operations.

So we have proved the first part of this theorem. To prove the second part
of this theorem, assume that $M = \jac H$ for a polynomial
map $H$. Just as in the proof of theorem \ref{quadrk1}, we may assume
that $H(0) = 0$ and $\deg H \le 2$. 
Say that $\tilde{M} = SMT$ for invertible matrices $S, T$ over $K$,
and let $\tilde{H} := SH(Tx)$.
The case where $\tilde{M}$ is as in (3) follows in a similar manner as the 
case where $\tilde{M}$ is as in (1) in the proof of theorem \ref{quadrk1}.

Hence assume that $\tilde{M}$ is as in (4) and that $\frac12 \in K$. 
Then $\jac \tilde{H}$ is as in (4) as well, and $\deg \tilde{H} \le 2$. 
From lemma \ref{antisym}, it follows that $\deg \tilde{H} = 1$. 
This contradicts that $\tilde{M}$ has
three entries which are linearly independent over $K$.
\end{proof}

\begin{corollary} \label{rk2trdeg2}
Let $H$ be a quadratic polynomial map over $K$, such that 
$r:=\rk \jac H \le 2$. If $\frac12 \in K$, then 
$K[H] \subseteq K[f_1,\ldots,f_r]$ for polynomials $f_i$. 
In particular, $\rk \jac H = \trdeg_K K(H)$.
\end{corollary}

\begin{proof}
From $\frac12 \in K$, it follows that for every term $t$ of $H$,
$$
\frac{\partial}{\partial x_i} t \ne 0 \Longleftrightarrow x_i \mid t
$$
If $r = 0$, then $H$ is constant. If $r = 1$, then it follows from
theorem \ref{quadrk1} that we may assume that either $K[H] = K[x_1]$
or $K[H] = K[H_1]$. So assume that $r = 2$. Then it follows from
theorem \ref{quadrk2} that we may assume that either $K[H] \subseteq K[x_1,x_2]$ 
or $K[H] = K[H_1,H_2]$ or $K[H] = K[H_1,x_1]$, because (4) of 
theorem \ref{quadrk2} for $M = \jac H$ requires $\frac12 \notin K$.
\end{proof}

\begin{lemma} \label{sym}
Let $H$ be a polynomial map over $K$ and suppose that $\jac H$ is symmetric. 
If for each $i$, the $i$\oth entry of the 
diagonal of $\jac H$ has no terms whose degrees with respect to $x_i$ are equal
to $-2$ in $K$, then there exists a polynomial $h \in K[x]$ such that 
$H = (\jac h)\tp$ and $\jac H = \hess h$.
\end{lemma}
 
\begin{proof}
Assume that the diagonal of $\jac H$ is as indicated above. Then for each $i$,
$H_i$ has no terms whose degrees with respect to $x_i$ are equal
to $-1$ in $K$. From the proof of \cite[Lem.~1.3.53]{MR1790619}, it follows
that there exists a polynomial $h \in K[x]$ such that $H = (\jac h)\tp$
(the $\alpha_i$ in that proof are nonzero). So $\jac H = \hess h$. 
\end{proof}
 
\begin{corollary} \label{rk2trdeg3}
Suppose that $H$ is a polynomial map of degree at most $2$ in dimension $3$
over $K$, such that $\jac H$ is antisymmetric with only zeroes on 
the diagonal. Suppose that $\jac H$ is not constant. Then there 
exist a $\lambda \in K^{*}$ and $c_1,c_2,c_3 \in K$, such that
$$
\jac H = \hess \big(\lambda(x_1+c_1)(x_2+c_2)(x_3+c_3)\big)
$$
Furthermore, $\frac12 \notin K$ and $\rk \jac H = 2 < 3 = \trdeg_K K(H)$.
\end{corollary}

\begin{proof}
From lemma \ref{antisym}, it follows that $\frac12 \notin K$ and that 
$\jac H$ is symmetric. From lemma \ref{sym}, it follows that
$\jac H = \hess h$ for some polynomial $h$.

Since $\deg H \le 2$, it follows that terms of degree greater than $3$ of $h$
cannot affect $\hess h$. Since $\frac12 \notin K$, it follows that terms of 
degree at most $3$ of $h$ which are divisible by $x_i^2$ for some $i$ cannot affect 
$\hess h$. So we can remove terms of $h$ of degree greater than $3$
and terms of $h$ which are divisible by $x_i^2$ for some $i$. Furthermore, we
can remove terms of $h$ of degree at most $1$. After these removals, $h$ will 
be of the form
$$
h = \lambda x_1x_2x_3 + \tilde{c}_1 x_2 x_3 + \tilde{c}_2 x_3 x_1 + \tilde{c}_3 x_1 x_2
$$
In particular $\deg h \le 3$. Suppose that $\jac H$ is not constant. Then $\deg h = 3$,
so $\lambda \ne 0$. Hence $\jac H$ is of the given form, with $c_i = \lambda^{-1} \tilde{c}_i$
for each $i$. Furthermore, $\rk \jac H = 2$.

Suppose that $\trdeg_K K(H) \le 2$. Then there exists a polynomial $f$ such
that $f(H) = 0$. If $\bar{f}$ is the leading homogeneous part of $f$ and
$\bar{H}$ is the leading homogeneous part of $H$, then $\bar{f}(\bar{H}) = 0$,
so $\trdeg_K K(\bar{H}) \le 2$. From \cite[Th.~2.7]{1501.06046}, it follows that 
there exists an $S \in \GL_3(K)$ such that 
$$
S\bar{H} = (p,q,0) \qquad \mbox{or} \qquad S\bar{H} = (p^2,pq,q^2)
$$
for homogeneous polynomials $p,q$ of the same degree.
In the first case, the rows of $\jac \bar{H}$ are dependent over $K$. 
In the second case, $\deg (p,q) = 1$ and the columns of $\jac \bar{H}$ 
are dependent over $K$. This is however not the case for
$$
\jac \bar{H} = \lambda \left( \begin{array}{ccc} 0 & x_3 & x_2 \\
               x_3 & 0 & x_1 \\ x_2 & x_1 & 0 \end{array} \right)
$$
so $\trdeg_K K(H) = 3$.
\end{proof}

\begin{corollary}
Let $M$ be a matrix whose entries are polynomials of degree 
at most $1$ over $K$. Suppose that $\rk M \le 2$ and that $M$ is the Jacobian matrix
of a polynomial map $H$. Then it is impossible to choose $H$ such that
$ \rk \jac H = \trdeg_K K(H)$ (and $\jac H = M$), if and only if
$M$ is as in {\upshape(4)} of theorem {\upshape\ref{quadrk2}}. 
In that case, $\frac12 \notin K$ and $\rk \jac H = 2 < 3 = \trdeg_K K(H)$ 
for every polynomial map $H$ such that $\jac H = M$.
\end{corollary}

\begin{proof}
The `if'-part follows from corollary \ref{rk2trdeg3}. The last claim 
follows from corollary \ref{rk2trdeg3} as well. The `only if'-part
follows from theorem \ref{quadrk2} and the proof of corollary 
\ref{rk2trdeg2}.
\end{proof}

\section{Nilpotent Jacobian matrices of degree one and rank at most two} \label{srk2n}

Before we prove the main result of this section, which is theorem \ref{srk2nmain} below,
we formulate a lemma about nilpotent matrices $N$ of degree $1$ and size $2 \times 2$ 
or $3 \times 3$, for which $N(0)$ has a simple structure.

\begin{lemma} \label{quadnil}
Let $K$ be a field and $N$ be a nilpotent matrix whose entries are polynomials 
of degree $1$. Then the following holds.
\begin{enumerate}[\upshape\bf (i)]

\item If $$N(0) = \left( \begin{array}{cc} 0 & 0 \\ 0 & 0 \end{array} \right)$$
then $N$ is similar over $K$ to a triangular matrix with only zeroes on the diagonal.

\item If $$N(0) = \left( \begin{array}{cc} 0 & 0 \\ 1 & 0 \end{array} \right)$$
then $N$ is lower triangular with only zeroes on the diagonal.

\item If $$N(0) = \left( \begin{array}{ccc} 0 & 0 & 1 \\ 0 & 0 & 0 \\ 0 & 0 & 0 
\end{array} \right)$$
then $N$ is similar over $K$ to a triangular matrix with only zeroes on the diagonal.

\item If $$N(0) = \left( \begin{array}{ccc} 0 & 1 & 0 \\ 0 & 0 & 1 \\ 0 & 0 & 0 
\end{array} \right)$$
and $N$ is \emph{not} upper triangular, then $N$ is similar over $K$ to a matrix of 
the form 
\begin{equation} \label{Nconj}
\left( \begin{array}{ccc} 0 & f+1 & 0 \\ b & 0 & f+1 \\ 0 & -b & 0 \end{array} \right)
\end{equation}
where $f$ and $b$ are linear forms and $b \ne 0$.

\end{enumerate}
\end{lemma}

\begin{proof} Since $N$ is nilpotent, it follows from lemma \ref{eigen}
that for every $r \le n$,  the sum of the principal minor determinants of
size $r \times r$ is zero. This gives us the trace condition if
$r = 1$, the $2 \times 2$ principal minors condition if $r = 2$, and the determinant
condition if $r = n$.
\begin{enumerate}[\upshape\bf (i)]

\item Using the trace condition, we see that 
$$
N = \left( \begin{array}{cc} a & b \\ c & -a \end{array} \right)
$$
for linear forms $a,b$. Using the determinant condition, we see that $bc = a^2$, so
$b \mid a^2$. As polynomial rings have unique factorization, $b = \lambda a$ for some 
$\lambda \in K$. Since $c = \lambda^{-1} a$, the entries of $a^{-1} N$
are contained in $K$. Hence $a^{-1} N$ is strongly nilpotent over an infinite extension field
of $K$. From proposition \ref{stronglt}, it follows that $a^{-1} N$ is similar over $K$ to 
a triangular matrix with only zeroes on the diagonal, and so is $N$.

\item Using the determinant condition,
$$
N = \left( \begin{array}{cc} {*} & 0 \\ {*} & {*} \end{array} \right)
$$
(because the linear part is zero). Hence $N$ is lower triangular. From 
lemma \ref{strongblock}, it follows that $N$ has only zeroes on the diagonal.

\item Using the principal $2 \times 2$ minors condition, we see that
$$
N = \left( \begin{array}{ccc} {*} & {*} & {*} \\ b & {*} & {*} \\ 
    0 & c & {*} \end{array} \right)
$$
for linear forms $b,c$ (because the linear part is zero). On account of the determinant 
condition, $bc = 0$ (because $bc$ is the quadratic part of $\det N$). 

If $b = 0$, then the leading principal $1 \times 1$ minor matrix and the trailing principal 
$2 \times 2$ minor matrix of $N$ are nilpotent on account of lemma \ref{strongblock}.
If $c = 0$, then the leading principal $2 \times 2$ minor matrix and the trailing principal 
$1 \times 1$ minor matrix of $N$ are nilpotent on account of lemma \ref{strongblock}.

In both cases, the principal $1 \times 1$ minor matrix is zero, because $\Mat_1(K)$ is a
reduced ring. From (i), we deduce that in both cases, the principal 
$2 \times 2$ minor matrix is similar over $K$ to a triangular matrix with only
zeroes on the diagonal. On account of corollary \ref{blocklt}, 
$N$ is similar over $K$ to a triangular matrix with only zeroes on the diagonal. 

\item Using all principal minors conditions
$$
N = \left( \begin{array}{ccc} -a & {*} & {*} \\ b & a+c & {*} \\ 
    0 & -b & -c \end{array} \right)
$$
for linear forms $a$, $b$, $c$ (because the linear parts are zero). 
If $b = 0$, then $N$ is upper triangular, so assume that $b \ne 0$. 
On account of the determinant condition, $a=c$ (quadratic part)
and $b \mid a(a+c)c$ (cubic part). So $b \mid a$ if $\frac12 \in K$. 
If $\frac12 \notin K$, then $b \mid a^2$ on account of the principal $2 \times 2$ 
minors condition, so $b \mid a$ in any case.

Since $b \mid a$ and $b \mid c$, we may assume that $a = c = 0$, because we can 
replace $N$ by $T^{-1}NT$, where
$$
T := \left( \begin{array}{ccc} 1 & -\frac{a}{b} & 0 \\ 0 & 1 & -\frac{c}{b} \\
     0 & 0 & 1 \end{array} \right)  \qquad \mbox{and} \qquad
T^{-1} = \left( \begin{array}{ccc} 1 & \frac{a}{b} & \frac{ac}{b^2} \\ 
         0 & 1 & \frac{c}{b} \\ 0 & 0 & 1 \end{array} \right)
$$
On account of the determinant condition, the upper right corner of $N$ is zero. 
On account of the principal $2 \times 2$ minors condition, $N$ is of 
the form of \eqref{Nconj} for some linear form $f$.
\qedhere

\end{enumerate}
\end{proof}

A square matrix over $K$ is similar over $K$ to its so-called Jordan
normal form, if and only if all its eigenvalues are contained in $K$.
Hence it follows from lemma \ref{eigen} that every nilpotent matrix
over $K$ is similar over $K$ to its Jordan normal form.

This fact about Jordan normal forms is used in the proof of theorem 
\ref{srk2nmain} below. The reader who is not familiar with Jordan normal forms
has to show some required similarities by hand, using the result of
proposition \ref{stronglt} that nilpotent matrices over $K$ are similar
over $K$ to triangular matrices with zeroes on their diagonals.

\begin{theorem} \label{srk2nmain}
Suppose that $H$ is a quadratic polynomial map in dimension $n$ over a field $K$ 
of any characteristic, such that $\rk \jac H \le 2$ and 
$\jac H$ is nilpotent. Then $\jac H$ is similar over $K$ to a triangular matrix. 
\end{theorem}

\begin{proof}
Let $M = \jac H$.
Suppose first that $\rk M = 1$. From theorem \ref{quadrk1}, it follows that
there exists $S,T \in \GL_n(K)$, such that $\tilde{M} := S M T$ satisfies one 
of the following:
\begin{itemize}

\item \emph{$\tilde{M}$ is as in {\upshape(1)} of theorem {\upshape\ref{quadrk1}}.}

Then $\tilde{M}$ and $T^{-1} M T$ are lower triangular, because only their first 
columns are nonzero. So $M$ is similar over $K$ to a triangular matrix.

\item \emph{$\tilde{M}$ is as in {\upshape(2)} of theorem {\upshape\ref{quadrk1}}.}

Then $\tilde{M}$ and $S M S^{-1}$ are upper triangular, because only their first 
rows are nonzero. So $M$ is similar over $K$ to a triangular matrix.

\end{itemize}
Suppose next that $\rk M = 2$. From theorem \ref{quadrk2}, it follows that
there exists $S,T \in \GL_n(K)$, such that $\tilde{M} := S M T$ satisfies one 
of the following:
\begin{itemize}

\item \emph{$\tilde{M}$ is as in {\upshape(1)} of theorem {\upshape\ref{quadrk2}}.}

Then only the first two columns of $\tilde{M}$ and $T^{-1} M T$ are nonzero.
On account of lemma \ref{strongblock}, the leading principal $2 \times 2$ minor 
matrix $N$ of $T^{-1} M T$ is nilpotent.

Since the Jordan normal form of $N(0)$ is equal to (that of) $N(0)$ 
in (i) or (ii) of lemma \ref{quadnil}, it follows from
(i) and (ii) of lemma \ref{quadnil} that $N$ is 
similar over $K$ to a triangular matrix with only zeroes on the diagonal.
From corollary \ref{blocklt}, we deduce that $M$ is similar over $K$ to a 
triangular matrix as well.

\item \emph{Only the first three rows of $\tilde{M}$ may be nonzero.}

Then only the first three rows of $\tilde{M}$, $S M S^{-1}$ and 
$\jac \big(S H(S^{-1}x)\big) = S M|_{x=S^{-1}x} S^{-1}$ are nonzero.
On account of lemma \ref{strongblock}, the leading principal $3 \times 3$ 
minor matrix $N$ of $S M S^{-1}$ is nilpotent. 

In order to show that $M$ is similar over $K$ to a triangular matrix, it suffices to
show that $S M S^{-1}$ is similar over $K$ to a triangular matrix. From corollary
\ref{blocklt}, we deduce that it suffices to show that $N$ is similar over $K$ to a 
triangular matrix with only zeroes on the diagonal. For that purpose, we distinguish 
three cases.
\begin{description}

\item[\mathversion{bold}$\rk N(0) = 0$.] Then we can replace $M$ 
by the result of substituting $x_i = x_i + 1$ in $M$ for some $i$, 
to obtain $\rk N(0) \ne 0$, because of the following.
$M$ becomes $\jac (H|_{x_i=x_i+1})$, which is a Jacobian matrix 
as well, and the linear part of $N$ is not affected. 

So if $N$ is similar over $K$ to a triangular matrix 
with only zeroes on the diagonal in the new situation, then 
$N$ is similar over $K$ to a triangular matrix with only zeroes 
on the diagonal originally.

\item[\mathversion{bold}$\rk N(0) = 1$.] Then the Jordan normal form 
of $N(0)$ is equal to that of $N(0)$ in (iii) of lemma \ref{quadnil}.
So we can choose $S$, such that $N(0)$ is as in (iii) of lemma \ref{quadnil}. 
It follows from (iii) of lemma \ref{quadnil} that $N$ is similar
over $K$ to a triangular matrix with only zeroes on the diagonal.

\item[\mathversion{bold}$\rk N(0) \ge 2$.] Then the Jordan normal form 
of $N(0)$ is equal to (that of) $N(0)$ in (iv) of lemma \ref{quadnil}. 
So we can choose $S$, such that $N(0)$ is as in (iv) of 
lemma \ref{quadnil}. 

Suppose first that $N$ is upper triangular. From lemma \ref{strongblock},
it follows that the leading principal minor matrix of size $1 \times 1$
and the trailing principal minor matrix of size $2 \times 2$ of $N$ are
nilpotent. By applying lemma \ref{strongblock} on the trailing 
principal minor matrix of size $2 \times 2$, we see that every 
principal minor matrix of size $1 \times 1$ is nilpotent. As $\Mat_1(K)$
is a reduced ring, the diagonal of $N$ is totally zero.

So assume that $N$ is not upper triangular. Then it follows from (iv) 
of lemma \ref{quadnil} that we can choose $S$, such that $N$ is of the 
form of \eqref{Nconj}, where $f$ and $b$ are linear forms and $b \ne 0$,
because $N(0)$ will not be affected. Define
$$
\tilde{N} := N \cdot \left( \begin{array}{ccc} 0 & 0 & -1 \\ 0 & -1 & 0 \\ 
    1 & 0 & 0 \end{array} \right) = \left( \begin{array}{ccc} 0 & -f-1 & 0 \\ 
    f+1 & 0 & -b \\ 0 & b & 0 \end{array} \right)
$$
Notice that $\tilde{N}$ is antisymmetric.
By definition of $\tilde{N}$ and $\tilde{M}$, we can choose $T$ such that 
$\tilde{N}$ is the leading principal $3 \times 3$ minor matrix of $\tilde{M}$.

Since the leading principal $2 \times 2$ minor matrix of $\tilde{N}(0)$ 
has full rank, we can clean the parts outside $\tilde{N}(0)$ of the first two 
rows of $\tilde{M}(0)$ by way of column operations. In other words, we can 
choose $T$ such that the submatrix of the first two rows of $\tilde{M}(0)$
equals
$$
\left(\begin{array}{ccc|ccc}
0 & -1 & 0 & 0 & \cdots & 0 \\
1 & 0 & 0 & 0 & \cdots & 0 
\end{array} \right)
$$
Looking at the constant parts of the $3 \times 3$ minor determinants,
we see that the third and subsequent entries of the third row of $\tilde{M}(0)$ 
are zero. 

Looking at the linear parts of the $3 \times 3$ minor determinants,
we see that the second entry of the third row of $\tilde{M}$ is the only nonzero 
entry in that row. So the third row of $\tilde{M}$ is of the form
$$
\left(\begin{array}{ccc|ccc}
0 & b & 0 & 0 & \cdots & 0
\end{array} \right)
$$
On account of $\eqref{chain}$, $\jac \tilde{H} = \tilde{M}|_{x=Tx}$.
Hence the second entry $\tilde{b}$ of the third row of $\jac \tilde{H}$ is the
only nonzero entry in that row. Furthermore, $\tilde{b}$ is a nonzero linear form
just like $b$. 

Consequently, $\tilde{b} = \lambda x_2$ for some nonzero $\lambda \in K$.
So the coefficient of $x_2^2$ in $\tilde{H}_3$ equals $\frac12 \lambda$.
In particular, 
$$
\tfrac12 \in K \qquad \mbox{and} \qquad \deg_{x_2} \tilde{H}_3 > 1
$$
Since $\jac_{x_1,x_2,x_3}(\tilde{H}_1,\tilde{H}_2,\tilde{H}_3) = 
\tilde{N}|_{x=Tx}$
is antisymmetric, it follows from lemma \ref{antisym} that 
$$
\deg_{x_2} \tilde{H}_3 \le 
\deg_{x_1,x_2,x_3} (\tilde{H}_1,\tilde{H}_2,\tilde{H}_3) \le 1
$$
Contradiction, so $N$ is upper triangular. \qedhere

\end{description}
\end{itemize}
\end{proof}

In the proof of \cite[Lem.~4]{MR3327129}, it is shown that $\jac H^2 = 0$ implies
$(\jac H)(x) \cdot (\jac H)(y) = 0$ if $H$ is quadratic homogeneous
and $\frac12 \in K$, where $y = (y_1,y_2,\ldots,\allowbreak y_n)$ is another $n$-tuple 
of indeterminates. The maps
$$
H = \big(0,x_1,x_1^2,x_1x_2-\tfrac12 x_3\big)
$$ 
and 
$$
H = \big(0,0,0,x_2 x_3, x_3 x_1, x_1 x_2, x_1 x_4 + x_2 x_5 + x_3 x_6\big)
$$
show that the conditions that $H$ is (quadratic) homogeneous and
$\frac12 \in K$ are necessary respectively.

\begin{theorem} \label{JH20}
Suppose that $H$ is a quadratic polynomial map in dimension $n$ over a field $K$ 
of any characteristic, such that $\jac H^2 = 0$. Then the following holds.
\begin{enumerate}[\upshape\bf (i)]

\item $\jac H$ is similar over $K$ to a triangular matrix.

\item If $\frac12 \in K$ and $H$ is homogeneous, then 
$$
(\jac H)(x)\cdot(\jac H)(y)=0
$$
where $y = (y_1,y_2,\ldots,y_n)$ is another $n$-tuple of indeterminates.

\item If $\frac12 \in K$ and $H$ is not (necessarily) homogeneous, then 
$$
(\jac H)(x)\cdot(\jac H)(y)\cdot(\jac H)(z)=0
$$
where $z = (z_1,z_2,\ldots,z_n)$ is yet another $n$-tuple of indeterminates.

\end{enumerate}
\end{theorem}

\begin{proof}
Let $y = (y_1,y_2,\ldots,y_n)$ be an $n$-tuple of indeterminates.
From $\deg \jac H \le 1$, it follows that
$$
(\jac H)\big(tx+(1-t)y\big) = t(\jac H)(x) + (1-t)(\jac H)(y)
$$
Taking squares on both sides, we deduce that
$$
0 = t(1-t)\big((\jac H)(x)\cdot(\jac H)(y)+(\jac H)(y)\cdot(\jac H)(x)\big)
$$
Consequently,
\begin{equation} \label{anticomm}
(\jac H)(y)\cdot(\jac H)(x) = - (\jac H)(x)\cdot(\jac H)(y)
\end{equation}
\begin{enumerate}[\upshape\bf (i)]

\item Let
$$
Z = \left( \begin{array}{cccc}
Z_{11} & Z_{12} & \cdots & Z_{1n} \\
Z_{21} & Z_{22} & \cdots & Z_{2n} \\
\vdots & \vdots & \sdots\, & \vdots \\
Z_{n1} & Z_{n2} & \cdots & Z_{nn}
\end{array} \right)
$$
be a matrix of indeterminates. Since we can substitute elements of 
any field $L \supseteq K$ in the indeterminates of the matrix $(Z\,|\,x\,|\,y)$, 
it follows from proposition \ref{stronglt} that it suffices to show that
$$
\Big( \prod_{i=1}^n (\jac H)(Ze_i) \Big)
\cdot (\jac H)(x) \cdot (\jac H)(y) = 0
$$
Just as in the proof of (iii) below, we can reduce to the case where $H$ is
homogeneous, at the cost of losing $(\jac H)(y)$ as a factor. So if we show that
\begin{equation} \label{Zx}
\Big( \prod_{i=1}^n (\jac H)(Ze_i) \Big) \cdot (\jac H)(x) = 0
\end{equation}
then we may assume that $H$ is homogeneous.

Hence suppose that $H$ is homogeneous. From (ii) of \cite[Prop.~3]{MR3327129} 
and \eqref{anticomm}, it follows that 
$$
\Big( \prod_{i=1}^n (\jac H)(Ze_i) \Big) \cdot (\jac H)(x) \cdot Ze_j
= \big(\,\sdots\,\big) \cdot \big((\jac H)(Ze_j)\big)^2 \cdot x = 0
$$
for every $j \le n$. Consequently
$$
\Big( \prod_{i=1}^n (\jac H)(Ze_i) \Big) \cdot (\jac H)(x) \cdot Z = 0
$$
As $\rk Z = n$, \eqref{Zx} follows.

\item This is shown in the proof of \cite[Lem.~4]{MR3327129}.

\item Let $\bar{H}$ be the quadratic part of $H$. Notice that
\begin{equation} \label{expand}
\prod_{i=1}^3 (\jac H)(Ze_i)
= \prod_{i=1}^3\left((\jac \bar{H})(Ze_i) + (\jac H)(0)\right) 
\end{equation} 
and that every term of the expansion of the right hand side of
\eqref{expand} either has two factors $(\jac H)(0)$, or two distinct 
(but not necessarily different) factors 
$(\jac \bar{H})(Ze_i)$ and $(\jac \bar{H})(Ze_j)$, where $1 \le i \le j \le 3$. 

From \eqref{anticomm}, it follows that
\begin{align*}
\big((\jac\bar{H})(x)\big)^2 
&= \big((\jac H)(x) - (\jac H)(0)\big)^2 \\
&= \big((\jac H)(x)\big)^2 + \big((\jac H)(0)\big)^2 = 0
\end{align*}
From (ii), we subsequently deduce that
\begin{equation} \label{anticommh}
(\jac\bar{H})(x) \cdot (\jac\bar{H})(y) = 0
\end{equation}
Furthermore, it follows from \eqref{anticomm} that
\begin{align*}
(\jac\bar{H})(x) \cdot (\jac H)(0) 
&= \big((\jac H)(x) - (\jac H)(0)\big) \cdot (\jac H)(0) \\
&= - (\jac H)(0) \cdot \big((\jac H)(x) - (\jac H)(0)\big) \\
&= - (\jac H)(0) \cdot (\jac \bar{H})(x)
\end{align*}
Consequently, the factors of the terms of the expansion of the right hand side of
\eqref{expand} anticommute. From \eqref{anticommh} and $\big((\jac H)(0)\big)^2 = 0$, 
we deduce that every term of the expansion of the right hand side of
\eqref{expand} equals zero. 
So $(\jac H)(Ze_1)\cdot(\jac H)(Ze_2)\cdot(\jac H)(Ze_3) = 0$. \qedhere

\end{enumerate}
\end{proof}

The conclusions of (ii) and (iii) of theorem \ref{JH20} can be reformulated as 
properties of a triangular matrix to which $JH$ is similar over $K$, see 
\cite[Th.~2.1]{MR3177043} and \cite[Cor.~2.2]{MR3177043}. Using this reformulation
more generally, one can deduce the following. 

\begin{proposition}
Let $H$ be a polynomial map, such that $\jac H$ is similar over $K$ 
to a triangular matrix. If $\jac H$ is nilpotent and $r = \rk \jac H$,
then 
$$
(\jac H) (Z e_1) \cdot (\jac H) (Z e_2) \cdot \cdots 
\cdot (\jac H) (Z e_r) \cdot (\jac H) (Z e_{r+1}) = 0
$$
where $Z$ is as in the proof of {\upshape(i)} of theorem {\upshape\ref{JH20}}.
\end{proposition}

It follows that the conclusions of (ii) and (iii) of theorem \ref{JH20} can be 
added to the cases $\rk \jac H = 1$ and $\rk \jac H = 2$ of theorem 
\ref{srk2nmain} respectively as well.

%\cite{DBLP:conf/mfcs/2016}

\bibliographystyle{quadrk2j}
\bibliography{quadrk2j}

\end{document}